\documentclass[12pt]{amsart}
\usepackage[applemac]{inputenc}

\usepackage{amssymb}
\newtheorem{prop}{Proposition}
\newtheorem{lemma}[prop]{Lemma}
\newtheorem{cor}[prop]{Corollary}
\newtheorem{thm}[prop]{Theorem}

\theoremstyle{definition}
\newtheorem*{defi}{Definition}

\newcommand{\ee}{\textrm{e}}
\begin{document}
\noindent

\title{Explicit constructions of extractors and expanders}

\author[]{Norbert Hegyv\'ari}
\address{Norbert Hegyv\'{a}ri, ELTE TTK,
E\"otv\"os University, Institute of Mathematics, H-1117
P\'{a}zm\'{a}ny st. 1/c, Budapest, Hungary}
\email{hegyvari@elte.hu}

\author[]{Fran\c cois Hennecart}
\address{Fran\c cois Hennecart,
Universit\'e de Lyon, F-69000, Lyon, France ;
Universit\'e de Saint-Etienne, F-42000, Saint-Etienne, France ; 
LAMUSE, 23 rue Michelon,
42023 Saint-Etienne, France} \email{francois.hennecart@univ-st-etienne.fr}

\thanks{Research of the first author
is partially supported by ``Balaton Program Project" and OTKA
grants K~61908, K~67676.
}
\date{\today}

%







%

\maketitle

\section{\bf Introduction}

The well-known Cauchy-Davenport theorem states that for any pair of sets
$A,B$ in $\mathbb{Z}_p$ such that $A+B\ne\mathbb{Z}_p$, we have
$|A+B|\geq |A|+|B|-1$ and this estimation
is sharp; for arithmetic progressions $A$, $B$ with common difference yield
$|A+B|=|A|+|B|-1$. Now a natural question arises; what can we say on the
image of a two variables (or more generally multivariable) polynomial.
One can ask which polynomial $f$ blows up its domain, i.e. if for any
$A,B\subseteq \mathbb{Z}_p,$ $|A|\asymp |B|$ then
$f(A,B):=\{f(a,b):a\in A; b\in B\}$ is ampler (in some uniform meaning) than $|A|$. As we
remarked earlier, the polynomial $f(x,y)=x+y$ is not admissible.

Let us say  that a polynomial $f(x,y)$ is an {\it expander} if $|f(A,B)|/|A|$
tends to infinity as $p$ tends to infinity (a more precise definition will be given
above).

According to the literature, very few is known about existence and construction of expanders; the only known explicit
construction is due to J.~Bourgain (see \cite{B}) who proved that the
polynomial $f(x,y)=x^2+xy$ is an expander. More precisely he proved
that if $p^\varepsilon <|A|\asymp |B|< p^{1-\varepsilon}$ then
$|f(A,B)|/|A|>p^{\gamma}$, where $\gamma=\gamma(\varepsilon)$ is a positive but inexplicit
real number.

Our aim is to extend of the class of
known expanders and to give some effective estimations for
$|f(A,B)|/|A|$. In particular in section \ref{s3} we will exhibit an infinite family of
two variables polynomials being expanders. The main tool is some
incidence inequality that will be also used  to construct explicit
extractors with three variables. A function $f:\mathbb{Z}^{3}\to\{-1,1\}$ is said to be
a $3$-source \textit{extractor} if under a certain condition on the size of $A,B,C$, the
sum $\sum_{(a,b,c)\in A\times B\times C}f(a,b,c)$ is small compared to the number of its terms
(see section \ref{newsection} for a sharp definition and the details).

Finally in the last section we show that extractors are connected with some additive questions.

\section{\bf Incidence inequalities for points and hyperplanes}

For  any prime number $p$,  we  denote by $\mathbb{F}_{p}$ the
fields with $p$ elements. The main tool used by Bourgain in
\cite{B} for exhibiting expanding maps and extractors is the
following Szemer\'edi-Trotter type inequality:

\begin{prop}[Bourgain-Katz-Tao Theorem]\label{p1}
Let $\mathcal{P}$ and $\mathcal{L}$ be respectively a set of
points and a set of lines in $\mathbb{F}^2_{p}$ such that
$$
|\mathcal{P}|,|\mathcal{L}|<p^{\beta}
$$
for some $\beta$, $0<\beta<2$. Then
$$
|\{(P,L)\in\mathcal{P}\times\mathcal{L}\ :\ P\in L\}|\ll
p^{(3/2-\gamma)\beta}
\quad\text{(as $p$ tends to infinity)},
$$
for some $\gamma>0$ depending only on $\beta$.
\end{prop}

In this statement, $\gamma$ can be calculated in terms of $\beta$ from the proof,
but it would imply a cumbersome formula. We will need the following consequence:

\begin{lemma}\label{lem1}
Let $\mathcal{P}$ and $\mathcal{L}$ be respectively a set of
points and a set of lines in $\mathbb{F}^2_{p}$ such that
$
|\mathcal{L}|<p^{\beta}
$
for some $\beta$, $0<\beta<2$. Then
\begin{equation}\label{e0}
|\{(P,L)\in\mathcal{P}\times\mathcal{L}\ :\ P\in L\}|\ll
|\mathcal{P}|^{3/2-\gamma'}+p^{(3/2-\gamma')\beta}
\quad\text{(as $p$ tends to infinity)},
\end{equation}
for some $\gamma'>0$ depending only on $\beta$.
\end{lemma}

\begin{proof} We denote by $N(\mathcal{P},\mathcal{L})$ the left-hand side of
\eqref{e0}.

We may freely assume that in Proposition \ref{p1}, 
\begin{equation}\label{e1}
\gamma=\gamma(\beta)<\frac{2-\beta}4.
\end{equation}
If $|\mathcal{P}|<p^{2-(2-\beta)/3}$, then the result follows plainly from Proposition \ref{p1}
with 
$$
\gamma'=\min(\gamma(\beta),\gamma(2-(2-\beta)/3)).
$$
Otherwise, we use the obvious bound $N(\mathcal{P},\mathcal{L})\le |\mathcal{L}|p<p^{1+\beta}$
from which we deduce
$$
N(\mathcal{P},\mathcal{L})<p^{(2-(2-\beta)/3)(3/2-\gamma)}\le |\mathcal{P}|^{3/2-\gamma}
$$
by \eqref{e1}. Thus \eqref{e0} holds with $\gamma'=\gamma$.
\end{proof}

In \cite{LAV}, the author established a generalization of Proposition \ref{p1}
by obtaining an incidence inequality for points an hyperplanes in $\mathbb{F}_{p}^{d}$.
It can be read as follows:

\begin{prop}[L.A. Vinh \cite{LAV}]\label{p2}
Let $d\ge2$. Let $\mathcal{P}$ be a set of points in $\mathbb{F}_{p}^{d}$ and $\mathcal{H}$ be a set of hyperplanes
in $\mathbb{F}_{p}^{d}$. Then
$$
\left|\{(P,H)\in \mathcal{P}\times \mathcal{H}\ :\ P\in H\}\right|
\le \frac{|\mathcal{P}||\mathcal{H}|}{p}+(1+o(1))p^{(d-1)/2}(|\mathcal{P}||\mathcal{H}|)^{1/2}.
$$
\end{prop}

From this, L.A. Vinh deduced in \cite{LAV} that in Proposition \ref{p1}, $\gamma$ can
be taken equal to $\frac{\min\{\beta-1;2-\beta\}}4$ whenever $1<\beta<2$.

\section{\bf A family of expanding maps of two variables}\label{s3}

For any prime number $p$, let $F_p:\mathbb{F}_p^k\to\mathbb{F}_p$
be an arbitrary function in $k$ variables in $\mathbb{F}_p$. One
says that the family of maps $F:=(F_p)_p$, where $p$ runs over the
prime numbers, is an \textit{expander} (in $k$ variables) if for
any $\alpha$, $0<\alpha <1$, there exist
$\epsilon=\epsilon(\alpha)>0$ such that for any positive real
numbers $L_1\le L_2$, and a positive constant $c=c(F,L_1,L_2)>0$
not depending on $\alpha$
 such that for any prime $p$ and for any $k$-tuples $(A_i)_{1\le i\le k}$  of subsets of $\mathbb{F}_{p}$ satisfying $L_1p^{\alpha}\le |A_i|\le L_2p^{\alpha}$ ($1\le i\le k$),
one has $|C_p|\ge c p^{\alpha+\epsilon}$
where $$C_p=F_p(A_1,A_2,\dots,A_k):=\{F_p(a_1,a_2,\dots,a_k)\ :\ (a_1,a_2,\dots,a_k)\in A_1\times A_2\times \cdots\times A_k\}.$$

If the maps $F_p$, $p$ prime, are induced by some function
$F:\mathbb{Z}^k\to\mathbb{Z}$, i.e. for any prime number $p$,
we have
$$
F_p(\pi_p(x_1),\dots,\pi_p(x_k))=\pi_p(F(x_1,\dots,x_k)),
$$
where $\pi_p$ is the canonical morphism from $\mathbb{Z}$ onto $\mathbb{F}_p$,
then we simply denote $F_p$ by $F$. If such $(F_p)_p$ is an expander, then
we will say that $F$ induces or is an expander.

For example, any integral polynomial function $F$ induces functions $F_p$ accordingly
denoted by $F$.
We will mainly concentrate our attention on the construction of expanders of this type.

In \cite{B}, the author proved that $F(x,y)=x^2+xy$ induces an expander and observed
that more general maps with two variables can be considered.
It is almost clear (see remark 1 in section \ref{s5}) that no map of the kind $f(x)+g(y)+c$ or $f(x)g(y)+c$
(where $c$ is a constant) can be an expander.
From this, one
deduces that maps of the type $F(x,y)=f(x)+(uf(x)+v)g(y)$ where $u,v\in\mathbb{F}_{p}$ and
$f$, $g$ are integral polynomials, are not expanders. It is clear if $u=0$, since in this case
$F(x,y)=f(x)+vg(y)$. If $u\ne0$, then  $F(x,y)=(f(x)+vu^{-1})(1+ug(y))-vu^{-1}$.
In order to exhibit expanders of the type $f(x)+h(x)g(y)$, we thus have
to assume that $f$ and $g$ are  affinely independent, namely there is no
 $(u,v)\in\mathbb{Z}^2$
such that $f(x)=uh(x)+v$ or $h(x)=uf(x)+v$.

We will show the following:

\begin{thm}\label{thmexp}
Let $k\ge1$ be an integer and $f$, $g$ be polynomials with integer coefficients, and
define for any prime number $p$, the map $F$ from
$\mathbb{Z}^2$ onto $\mathbb{Z}$  by
$$
F(x,y)=f(x)+x^kg(y)
$$
Assume moreover that  $f(x)$ is affinely independent to $x^k$.
Then
$F$ induces an expander.
\end{thm}

For $p$
sufficiently large, the image  $g(B)$ of any subset $B$ of
$\mathbb{F}_{p}$ has cardinality at least $|B|/\deg(g)$. It
follows that we can restrict our attention to maps of the type
$F(x,y)=f(x)+x^{k}y$. We let $d:=\deg(f)$.

Let $A$ and $B$ be subsets of $\mathbb{F}_{p}$ with cardinality
$|A|\asymp|B|\asymp p^{\alpha}$. For any $z\in\mathbb{F}_{p}$,
we denote by $r(z)$ the number
of couples $(x,y)\in A\times B$ such that $z=F(x,y)$, and by $C$
the set of those $z$ for which $r(z)>0$. By Cauchy-Schwarz inequality,
we get
$$
|A|^2|B|^2=\Big(\sum_{z\in\mathbb{F}_{p}}r(z)\Big)^2\le|C|\times
\Big(\sum_{z\in\mathbb{F}_{p}}r(z)^2\Big).
$$
One now deal with the sum $\sum_{z\in\mathbb{F}_{p}}r(z)^2$
which can be rewritten as the number of quadruples
$(x_{1},x_{2},y_{1},y_{2})\in A^{2}\times B^{2}$ such that
\begin{equation}\label{equat1}
f(x_{1})+x_{1}^{k}y_{1}=f(x_{2})+x_{2}^ky_{2}.
\end{equation}
For fixed $(x_{1},x_{2})\in A^2$ with
$x_{1}\ne0$ or $x_{2}\ne0$, \eqref{equat1} can be viewed as the equation of
a line $\ell_{x_{1},x_{2}}$ whose points $(y_{1},y_{2})$ are
in $\mathbb{F}_{p}^2$. For $(x_{1},x_{2})$ and $(a,b)$ in $A^2$,
the lines $\ell_{x_{1},x_{2}}$ and $\ell_{a,b}$ coincide if and only if
\begin{equation*}
\left\{
\begin{aligned}
(x_{1}b)^k&=(ax_{2})^k\\
{b}^k(f(x_{2})-f(x_{1}))&=x_{2}^k(f(b)-f(a)),
\end{aligned}
\right.
\end{equation*}
or equivalently
\begin{equation}\label{equat2}
\left\{
\begin{aligned}
(x_{1}b)^k&=(ax_{2})^k\\
(b^k-a^k)(f(x_{2})-f(x_{1}))&=(x_{2}^k-x_{1}^k)(f(b)-f(a)).
\end{aligned}
\right.
\end{equation}
At this point observe that by our assumption, there are only finitely many
prime numbers $p$ such that $f(x)=ux^k+v$
for some $(u,v)\in\mathbb{F}_{p}^2$, in which case
 the second equation in \eqref{equat2} holds trivially for any
 $x_{1}$ and $x_{2}$. We assume in the sequel that $p$ is not such a prime number.


Let $(a,b)\in A^2$ such that $a\ne0$ or $b\ne0$. Assume for instance that $b\ne0$.
By \eqref{equat2} we get $x_{1}=\frac{\zeta a x_{2}}{b}$
 for some $k$-th root modulo $p$ of unity $\zeta$. Moreover, we obtain
 \begin{equation}\label{equat3}
b^k\Big(f(x_{2})-f(\zeta \frac{a x_{2}}{b})\Big)
 -x_{2}^k(f(b)-f(a))=0,
 \end{equation}
which  is a polynomial equation in $x_{2}$. If we write
$f(x)=\sum_{0\le j \le d}f_{j}x^j$ then
$$
b^k(f(x)-f(\zeta \frac{ax}{b}))=\sum_{1\le j\le d}
b^k(1-\frac{\zeta^ja^j}{b^j})f_{j}x^j
$$
is a polynomial which could be identically equal
to $x^k(f(b)-f(a))$ only if the following  two conditions are satisfied:
\begin{align*}
&f(b)-f(a)=(b^k-a^k)f_{k},\\
&f_{j}\ne0\Rightarrow b^j=\zeta^ja^j.
\end{align*}
Since $f(x)$ is assumed to be affinely independent to $x^k$, we necessarily have $f_{j}\ne0$
for some $0<j\ne k$.
If $b^j=\zeta^ja^j$ for $\zeta$ being a $k$-th root of unity
in $\mathbb{F}_{p}$, then $b=\eta a$ where $\eta$ is some
$(kd!)$-root of unity in $\mathbb{F}_{p}$.
Let
$$
X:=\{(a,b)\in A^2\ :\ b^{kd!}\ne a^{kd!}\}.
$$
Since there are $kd!$ many
$(kd!)$-roots of unity in $\mathbb{F}_{p}$,
We have $|A^2\smallsetminus X|\le kd!|A|$, hence
$|X|\ge \frac{|A|^2}{2}$ for $p$ large enough.

If $(a,b)\in X$, then \eqref{equat3} has at most $\max(k,d)$ many solutions $x_{2}$,
thus $\eqref{equat2}$ has at most $k\max(k,d)$ many solutions $(x_{1},x_{2})$.
We conclude that the number of distinct lines $\ell_{a,b}$ when
$(a,b)$ runs in $ A^2$ is  $c(k,f)|A|^2$ where $c(k,f)$ can be chosen
equal to $(2k\max(k,d))^{-1}$, for $p$ large enough.
The set of all these pairwise distinct lines $\ell_{a,b}$ is denoted
by $\mathcal{L}$, its cardinality satisfies $|A|^2\ll |\mathcal{L}|\le|A|^2$,
as observed before.
Let $\mathcal{P}=B^2$. Then putting $N:=|A|^2\asymp|B|^2$,
we have by Proposition \ref{p1}
$$
\big\{(p,\ell)\in\mathcal{P}\times\mathcal{L}\ :\ p\in\ell\big\}\ll N^{3/2-\delta}
$$
for some $\delta>0$. Hence the number of solutions of
the system \eqref{equat2} is  $O(N^{3/2-\delta})=O(|A|^2|B|^{1-2\delta})$.
Finally $|C|\gg |B|^{1+2\delta}$, which is the desired conclusion.

\section{\bf Further results on expanders}

When $\alpha>1/2$, instead of Bourgain-Katz-Tao's incidence inequality,
we can use  Proposition \ref{p2}. By the remark following Proposition \ref{p2}, we can
replace in the very end of our proof of  Theorem \ref{thmexp}, $\delta$ by
$\min\{2\alpha-1;2-2\alpha\}$. It gives

\begin{prop} Let $F$ as in Theorem \ref{thmexp} and $\alpha>1/2$. For any
pair $(A,B)$ of subsets of $\mathbb{F}_p$ such that $|A|\asymp|B|\asymp p^{\alpha}$, we have
$$
|F(A,B)|\gg |A|^{1+\frac{\min\{2\alpha-1;2-2\alpha\}}{2}}.
$$
\end{prop}

The notion of expander which we discussed in the previous section is concerning the ability for a
two variables function $F$, inducing a sequence $(F_p)_p$, to provide a non trivial uniform lower bound for
$$
\kappa_{\alpha}(F)=\inf_{0<L_1<L_2}\liminf_{p\to\infty}\min\Big\{\frac{\ln|F_p(A,B)|}{\ln|A|}\,:\, A,B\subset\mathbb{F}_p\text{ and }
 L_1p^{\alpha}\le |A|,|B|\le L_2p^{\alpha}\Big\}
$$
For $F$ introduced in Theorem~\ref{thmexp}, we thus have
$$
1+ \frac{\min\{2\alpha-1;2-2\alpha\}}{2}\le   \kappa_{\alpha}(F)\le \min\{2,\frac1{\alpha}\},
$$
where the upper bound follows from the plain bounds $|F(A,B)|\le|A||B|$ and  $|F(A,B)|\le p$.
To our knowledge, no explicit example of function $F$ such that $\kappa_{\alpha}(F)=\min\{2,\frac1{\alpha}\}$ has been
already provided in the literature, even for a given real number $\alpha$ with $0<\alpha <1$. This question
is certainly much more difficult than the initial question of providing expander. This  suggests the
following definition:

\begin{defi}
Let $I\subset(0,1)$ be a non empty interval.
A family $F=(F_p)_p$ of two variables functions is called
\begin{itemize}
\item a \textit{strong expander according to $I$} if for any $\alpha\in I$, we have
$$
\kappa_{\alpha}(F)=\min\{2,\tfrac1{\alpha}\}. 
$$
\item a \textit{complete expander according to $I$} if for any $\alpha\in I$, for any positive real numbers $L_1\le L_2$, there exists a constant $c=c(F,L_1,L_2)$ such that
for any prime number $p$ and any pair $(A,B)$ of subsets of $\mathbb{F}_p$ satisfying
$L_1p^{\alpha}\le|A|,|B|\le L_2p^{\alpha}$,
we have
$$
|F_p(A,B)|\ge c p^{\min\{1;2\alpha\}}.
$$
\end{itemize}
\end{defi}
Complete expanders according to $I$ are obviously strong expanders
according to $I$. As indicated in \cite{B}, random mapping are strong expanders with
a large probability,
but no explicit example is known.  Furthermore functions $F$ introduced in Theorem
\ref{thmexp} could eventually be strong expanders, but we can not prove or disprove this
fact.  Nevertheless,
we can show that some of them are not complete expanders, in
particular Bourgain's function $F(x,y)=x^2+xy=x(x+y)$. Indeed, let
$A$ and $B$ be the interval $[1,p^{\alpha}/2]$ in $\mathbb{Z}_p$.
Then $A\cup(A+B)\subset[1,p^{\alpha}]$. If we assume $\alpha\le
1/2$, the following result which is a direct consequence of a result by
Erd\H os (see \cite{E2,E3}) implies that $F(A,B)=A\cdot(A+B)$ has cardinality at
most $o(p^{2\alpha})$.

\begin{lemma}[Erd\H os Lemma]
There exists a positive real number $\delta$ such that the number of different
integers $ab$ where $1\le a,b\le n$ is $O(n^2/(\ln n)^{\delta})$.
\end{lemma}

A sharper result due to G. Tenenbaum \cite{T} implies that $\delta$ can be taken equal to $1-\frac{1+\ln\ln 2}{\ln 2}$
in this statement.


In the same vein, we can extend Bourgain's result to more general functions:

\begin{prop} \label{p6}
Let $k\ge2$ be an integer, $u\in\mathbb{Z}$
and $F(x,y)=x^{2k}+ux^{k}+x^ky=x^k(x^k+y+u)$.
Then for any $\alpha$, $0<\alpha\le 1/2$, $F$ is not a complete expander according to
$\{\alpha\}$.
\end{prop}

\begin{proof}
Let $L$ be a positive integer such that $L<\sqrt{p}/2$.
The set of $k$-th powers in $\mathbb{F}_p^*$ is a subgroup of $\mathbb{F}_p^*$ with  index
$l=\gcd(k,p-1)\le k$. Thus there exists $a\in\mathbb{F}_p^*$ such that
$[1,L]$ contains at least $L/l$ residue classes of the form $ax^k$, $x\in\mathbb{F}_p^*$.
We let $A=\{x\in\mathbb{F}_p^*\,:\, ax^k\in[1,L]\}$, which has cardinality at least
$L$ since each $k$-th power has $l$ $k$-th roots modulo $p$. We let
$B=\{y\in\mathbb{F}_p\,:\, a(y+u)\in[1,L]\}$. We clearly have $|B|=L$.
Moreover the elements of $F(A,B)$ are of the form $x^k(x^k+y+u)$ with
$x\in A$ and $y\in B$, thus are of the form $a'^2x'y'$ where
$x',y'\in[1,2L]$ and $aa'=1$ in $\mathbb{F}_p$.
By Erd\H os Lemma, we infer $|F(A,B)|=O(L^2/(\ln L)^{\delta})=o(L^2)$.
\end{proof}

By using a deep bound by Weil on exponential sums with polynomials, we may
slightly extend this result:

\begin{prop} \label{p6b}
Let $f(x)$ and $g(y)$ be  non constant integral polynomials
and $F(x,y)=f(x)(f(x)+g(y))$.
Then $F$ is not a complete expander according to
$\{1/2\}$.
\end{prop}

We shall need the following result:

\begin{lemma} Let $u\in\mathbb{F}_p$, $L$ be a positive integer
less than $p/2$ and $f(x)$ be any integral polynomial of degree $k\ge1$ (as element of $\mathbb{F}_p[x]$).
Then the number $N(I)$ of
residues $x\in\mathbb{F}_p$ such that $f(x)$ lies in the interval  $I=(u-L,u+L)$ of $\mathbb{F}_p$
is at least $L-(k-1)\sqrt{p}$.
\end{lemma}

\begin{proof}
We will use the formalism of Fourier analysis. Recall the following notation and properties:

Let $\phi,\psi:\mathbb{F}_p\to \mathbb{C}$ and $x\in\mathbb{F}_p$.
\begin{itemize}
\item $\phi*\psi(x):=\sum_{y\in\mathbb{F}_p}\overline{\phi(y)}\psi(x+y)$;

\item $\hat{\phi}(x):=\sum_{y\in\mathbb{F}_p}\phi(y)\ee\big(\frac{yx}p\big)$, where
$\ee(t):=\exp(2i\pi t)$;

\item $\widehat{\phi*\psi}(x)=\overline{\hat{\phi}(x)}\hat{\psi}(x)$;

\item $\sum_{y\in\mathbb{F}_p}|\hat{\phi}(y)|^2=p\sum_{y\in\mathbb{F}_p}|\phi(y)|^2$
(Parseval's identity).

\end{itemize}

Let $J$ be the indicator function of the interval $[0,L)$ of $\mathbb{F}_p$ and let
$$
T:=\sum_{h\in\mathbb{F}_p}\widehat{J*J}(h) S_f(-h,p)\ee\Big(\frac{hu}p\Big),
$$
where the exponential sum
$$
S_f(h,p):=\sum_{x\in\mathbb{F}_p}\ee\Big(\frac{hf(x)}p\Big)
$$
is known to satisfy the bound
$|S_f(h,p)|\le (k-1)\sqrt{p}$ whenever $h\ne0$ in $\mathbb{F}_p$ and $p$ is an odd prime number
(see for instance \cite{V}).

On the one hand, we have
\begin{align*}
T&=p\widehat{J*J}(0)+\sum_{h\in\mathbb{F}_p\smallsetminus\{0\}}\widehat{J*J}(h) S_f(-h,p)
\ee\Big(\frac{hu}p\Big)\\
&\ge pL^2-k\sqrt{p}\sum_{h\in\mathbb{F}_p\smallsetminus\{0\}}|\widehat{J*J}(h)|\\
&\ge pL^2-kLp^{3/2},
\end{align*}
by the bound for Gaussian sums and Parseval Identity. Hence
\begin{equation}\label{eqn2}
T\ge pL(L-k\sqrt{p})
\end{equation}

On the other hand,
\begin{align*}
T&= \sum_{h\in\mathbb{F}_p}\sum_{y\in\mathbb{F}_p}\sum_{z\in\mathbb{F}_p}
J(z)J(y+z)\ee\Big(\frac{h(y+u)}{p}\Big)\sum_{x\in\mathbb{F}_p}\ee\Big(-\frac{hf(x)}p\Big)\\
&= \sum_{x\in\mathbb{F}_p}\sum_{y\in\mathbb{F}_p}\sum_{z\in\mathbb{F}_p}
J(z)J(y+z)\sum_{h\in\mathbb{F}_p}\ee\Big(\frac{h(y+u-f(x))}p\Big)\\
&= p\sum_{x\in\mathbb{F}_p}d_L(f(x)-u),
\end{align*}
where $d_L(z)$ denotes the number of representations in $\mathbb{F}_p$ of $z$ under the form $j-j'$, $0\le j,j'<L$.
Since obviously $d_L(z)\le L$ for each $z\in\mathbb{F}_p$, we get
$$
T\le pL N(I).
$$
Combining this bound and \eqref{eqn2}, we deduce the lemma.
\end{proof}

\begin{proof}[Proof of Propostion \ref{p6b}]
We choose $p$ large enough so that both $f(x)$ and $g(y)$ are not constant polynomials
modulo $p$. Let $L=k\sqrt{p}$, and define $A$ (resp. $B$) to
be the set of the residue classes $x$ (resp. $y$) such that $f(x)$ (resp. $g(y)$)
lies in the interval
$(0,2L)$. By the previous lemma, one has $|A|,|B|\ge\sqrt{p}$.
Moreover for any
$(x,y)\in A\times B$, we have $f(x)$ and $f(x)+g(y)$ in the interval $(0,4L)$.
By Erd\H os Lemma, the number of residues modulo $p$ which can be written as
$F(x,y)$ with $(x,y)\in A\times B$, is at most $O(L^2/(\ln L)^{\delta})=o(p)$, as $p$ tends to infinity.
\end{proof}


\section{\bf A family of $3$-source extractors with exponential distribution}\label{newsection}

Let us fix the definition of the \textit{entropy} of a
 $k$-source $f=(f_p)_p$ where $f_p:\mathbb{F}_p^{k}\to\{-1,1\}$ as follows :
it is defined to be
the infimum, denoted $\alpha_0$, on $\alpha>0$ such that for any subset $A_{j}$, $j=1, \dots,k$,
of $\mathbb{F}_{p}$ with cardinality at least $p^{\alpha}$,
we have
$$
\sum_{\substack{a_{j}\in A_{j}\\
j=1,\dots,k}}f_p(a_{1},\dots,a_{k})=o(\prod_{j=1}^k|A_{j}|),\qquad
\text{as $p\to+\infty.$}
$$
When $\alpha_0<1$, $f$ is called $k$-source \textit{extractor} (with entropy $\alpha_0$).

The problem of finding $k$-source extractors can be reduced as follows.
We are asking the question to find  functions $F_p:\mathbb{F}_p^k\to\mathbb{F}_p$ such that
for any $k$-tuples $(A_{1},A_{2},\dots,A_{k})$ of subsets of $\mathbb{F}_{p}$
with cardinality $\asymp p^{\alpha}$
such that for any $r\in\mathbb{F}_{p}^{\times}$
\begin{equation}\label{eqn1}
\Big|\sum_{\substack{a_{j}\in A_{j}\\
j=1,\dots,k}}\ee_{r}\left(F_p(x_{1},x_{2},\dots,x_{k})\right)\Big|
=O(p^{-\gamma}\prod_{j=1}^k|A_{j}|), \quad \text{as $p$ tends to infinity,}
\end{equation}
for some $\gamma=\gamma(\alpha)$ and where we denote
$\ee_{r}(u)=\exp(\frac{ru}{p})$. If \eqref{eqn1} holds, Bourgain (cf. \cite{B}) has shown that
\begin{equation}\label{eqn1p}
\sum_{\substack{a_{j}\in A_{j}\\
j=1,\dots,k}}f_p(a_{1},\dots,a_{k})=O(p^{-\gamma'}\prod_{j=1}^k|A_{j}|),\qquad
\text{as $p\to+\infty$}
\end{equation}
for some $\gamma'>0$ where $f_p:=\mathrm{sgn}\sin \frac{2\pi F_p}p$.
It thus gives  a $k$-source extractor $f=(f_p)_p$.
An extractor $f$ such that \eqref{eqn1p} holds is said to have an \textit{exponential
distribution}.

In \cite[Proposition 3.6]{B}, Bourgain proved that
$F(x,y)=xy+x^2y^2$, by letting $F=F_p$ for any $p$, provides a
$2$-source extractor with exponential distribution and with
entropy $1/2-\delta$ for some $\delta>0$. We will show that this
result can be extended in order to give $3$-source extractors with
such an  entropy. It has to be mentioned that explicit $3$-source
extractors with arbitrary positive entropy exists, as shown in
\cite{BKSSW}, but these extractors do not yield an exponential
distribution. Here our goal is to exhibit $3$-source extractors
with exponential distribution.

\begin{thm}\label{thmNewSE}
Let $F(x,y,z)=a(z)xy+b(z)x^2g(y)+h(y,z)\in\mathbb{Z}[x,y,z]$ where $a(z)$, $b(z)$
are any non zero polynomial function,
$g(y)$ is any polynomial function of degree at least two and $h(y,z)$ an arbitrary
polynomial function. Let $L_1\le L_2$ be positive real numbers,
$\alpha\in(0,1)$ and
$A,B,C$ be subsets of $\mathbb{F}_p$ with cardinality satisfying  $L_1p^{\alpha}\le |A|,|B|,|C|\le L_2p^{\alpha}$.
For $r\in\mathbb{F}_p$, we denote
$$
S_{r}=\sum_{(x,y,z)\in A\times B\times C}\ee_{r}\left(F(x,y,z)\right).
$$
Then there exists $\gamma=\gamma(\alpha)>0$ such that
$$
\max_{r\in\mathbb{F}_p\smallsetminus\{0\}}|S_r|\ll p^{((22-\gamma/2)\alpha+1)/8},
$$
where the implied constant depends only on $F$, $L_1$ and  $L_2$.
\end{thm}

\begin{proof}
The proof starts as in \cite[Proposition 3.6]{B}.
For any  $r\in\mathbb{F}_{p}\smallsetminus\{0\}$, let
$$
S_{r}=\sum_{(x,y,z)\in A\times B\times C}\ee_{r}\left(F(x,y,z)\right).
$$

The first transformations consist in using repeatedly Cauchy-Schwarz inequality
in order to increase the number of variables and to rely
$
S_{r}
$
to the number
of solutions of diophantine systems. We simply denote $S_{r}$ by $S$.
We denote by $C_{0}$ the subset of $C$ formed with the elements $z\in C$
such that $a(z)b(z)=0$. We let $C':=C\smallsetminus C_{0}$. Then
$S=S_{0}+S'$ where in $S_{0}$ (resp. $S'$) the summation over $z$ is restricted
$z\in C_{0}$ (resp. $z\in C'$). Since the number of roots of the equation
$a(z)b(z)=0$ is finite, we have $|S_{0}|\ll |A||B|\ll p^{2\alpha}$. Moreover
we get
\begin{align*}
|S'|&\le \sum_{y,z}\Big|\sum_{x} \ee_{r}\left(a(z)xy+b(z)x^2g(y)\right)\Big|\\
&\le \Big(\sum_{y,z}1\Big)^{1/2}\Big(\sum_{\substack{y,z\\x_{1},x_{2}}}
\ee_{r}\left(a(z)(x_{1}-x_{2})y+b(z)(x_{1}^2-x_{2}^2)g(y)\right)
\Big)^{1/2},
\end{align*}
where the summation over $z$ is restricted to $z\in C'$.
Hence
\begin{align*}
|S'|^2 &\ll p^{2\alpha}\sum_{y,z}\Big|\sum_{x_{1},x_{2}}
 \ee_{r}\left(a(z)(x_{1}-x_{2})y+b(z)(x_{1}^2-x_{2}^2)g(y)\right)
 \Big|\\
&\ll p^{2\alpha}\Big(\sum_{y,z}1\Big)^{1/2}
\Big(\sum_{\substack{x_{1},x_{2}\\
x_{3},x_{4}\\y,z}}
 \ee_{r}\big(a(z)(x_{1}-x_{2}+x_{3}-x_{4})y+
 b(z)(x_{1}^2-x_{2}^{2}+x_{3}^2-x_{4}^{2})g(y)\big)
\Big)^{1/2}
\end{align*}
then
\begin{equation*}
|S'|^4\ll p^{6\alpha}\sum_{\substack{x_{1},x_{2}\\
x_{3},x_{4}\\y,z}}
\ee_{r}\big(a(z)(x_{1}-x_{2}+x_{3}-x_{4})y+
 b(z)(x_{1}^2-x_{2}^{2}+x_{3}^2-x_{4}^{2})g(y)\big)
\end{equation*}
By a new application of Cauchy-Schwarz inequality, we get
\begin{align*}
|S'|^{8} & \ll p^{12\alpha}\Big(
\sum_{\substack{x_{1},x_{2}\\x_{3},x_{4}\\z}}
\Big|\sum_{y}\ee_{r}\big(a(z)(x_{1}-x_{2}+x_{3}-x_{4})y+
 b(z)(x_{1}^2-x_{2}^{2}+x_{3}^2-x_{4}^{2})g(y)\big)
\Big|\Big)^2\\
& \ll p^{17\alpha} \sum_{z}\sum_{\substack{x_{1},x_{2}\\
x_{3},x_{4}\\y_{1},y_{2}}}
\ee_{r}\big(a(z)(x_{1}-x_{2}+x_{3}-x_{4})(y_{1}-y_{2})\\
&\hspace*{7.5cm}+
 b(z)(x_{1}^2-x_{2}^{2}+x_{3}^2-x_{4}^{2})(g(y_{1})-g(y_{2}))\big)
 \\
 &= p^{17\alpha}\sum_{z}\sum_{\underline{\xi},\underline{\eta}\in\mathbb{F}_{p}^2}
 \mu(\underline{\xi})\nu(\underline{\eta})\ee_{r}\big(a(z)\xi_{1}\eta_{1}+
 b(z)\xi_{2}\eta_{2}\big)
 \end{align*}
 where $\mu(\underline{\xi})$ is the number of quadruples $(x_{1},x_{2},x_{3},x_{4})\in A^4$ such that
 \begin{equation}\label{equat4}
 \left\{
 \begin{aligned}
 \xi_{1}&=x_{1}-x_{2}+x_{3}-x_{4},\\
 \xi_{2}&=x_{1}^2-x_{2}^2+x_{3}^2-x_{4}^2,
 \end{aligned}
 \right.
 \end{equation}
 and $\nu(\underline{\eta})$ is the number of couples $(y_{1},y_{2})\in B^2$
 such that
 \begin{equation*}
 \left\{
 \begin{aligned}
 \eta_{1}&=y_{1}-y_{2},\\
 \eta_{2}&=g(y_{1})-g(y_{2}).
 \end{aligned}
 \right.
 \end{equation*}
  Then clearly $\sum_{\underline{\eta}\in\mathbb{F}_{p}^2}
  \nu(\underline{\eta})^2$
 can be expressed as the number of quadruples
 $(y_{1},y_{2},y'_{1},y'_{2})\in B^4$ such that
 \begin{equation}\label{equat5}
\left\{
 \begin{aligned}
 y_{1}-y_{2}&=y'_{1}-y'_{2},\\
  g(y_{1})-g(y_{2})&=g(y'_{1})-g(y'_{2}).
 \end{aligned}
 \right.
 \end{equation}
 If $y'_{1}=y'_{2}$ in this system then $y_{1}=y_{2}$. Thus \eqref{equat5}
 has exactly $|B|^2$ solutions of the type $(y_{1},y_{2},y'_{1},y'_{1})$.
 If $y'_{1}$ and $y'_{2}$ are fixed so that $t=y'_{1}-y'_{2}\ne0$, then we can write $y_{1}=y_{2}+t$ and
 and clearly $g(y_{2}+t)-g(y_{2})=g(y'_{1})-g(y'_{2})$ has at most $\deg g-1$
 solutions $y_{2}$ (since $\deg g\ge2$). We thus have
 \begin{equation}\label{equat6}
 \sum_{\underline{\eta}\in\mathbb{F}_{p}^2}
 \nu(\underline{\eta})^2\ll p^{2\alpha}.
 \end{equation}

For any $\underline{\xi}=(\xi_1,\xi_2)\in\mathbb{F}_p^2$, we denote by
$\mu_1(\underline{\xi})$ (resp. $\mu_2(\underline{\xi})$) the number of solutions $(x_1,x_2,x_3,x_4)\in A^4$ of \eqref{equat4}
such that $x_1=x_2$ (resp. $x_1\ne x_2$). Then
$$
\sum_{\underline{\xi}\in\mathbb{F}_p^2}
\mu_1(\underline{\xi})^2=|A|^2\times N,
$$
where $N$ is the number of quadruples $(x_3,x_4,z_3,z_4)\in A^4$ such that
$$
\left\{
 \begin{aligned}
 x_{3}-x_{4}&=z_3-z_4,\\
  x_3^2-x_4^2&=z_3^2-z_4^2.
 \end{aligned}
 \right.
$$
By distinguishing solutions with $x_3=x_4$ and solutions with $x_3\ne x_4$, we
plainly obtain $N\le 2|A|^2$. Hence
\begin{equation}\label{equatn1}
\sum_{\underline{\xi}\in\mathbb{F}_p^2}
\mu_1(\underline{\xi})^2\ll p^{4\alpha}.
\end{equation}

For any fixed $t\in A$,
we denote by $\mu(\underline{\xi},t)$ the number of solutions
of the form $(x_{1},x_{2},t,x_{4})\in A^4$ with $x_1\ne x_2$ of the system \eqref{equat4}.
Eliminating $x_{4}$ by expressing it in terms
of $\xi_{1}$ using the first equation, we see
that  $\mu(\underline{\xi},t)$ is the number of couples
$(x_{1},x_{2})\in A^2$ with $x_1\ne x_2$
such that $\underline{\xi}$ lies on the curve
\begin{equation}
\xi'_{2}:=\xi_{2}+\xi_{1}^2=
2(x_{1}-x_{2}+t)\xi_{1}-(x_{1}-x_{2}+t)^2
+x_{1}^2-x_{2}^2+t^2.\label{equat7}
\end{equation}
Using the new variable $\xi'_{2}$ instead of $\xi_{2}$,
we get that each couple
$(x_{1},x_{2})\in A^2$ with $x_1\ne x_2$ defines a line $\ell_{x_1,w_2}$
in the plane $\mathbb{F}_{p}^2$ with equation
\begin{equation}\label{equat8}
\xi'_{2}=2(x_{1}-x_{2}+t)\xi_{1}-(x_{1}-x_{2}+t)^2
+x_{1}^2-x_{2}^2+t^2.
\end{equation}
It is clear that two couples $(x_{1},x_{2})\in A^2$  and
$(x'_{1},x'_{2})\in A^2$ with $x_1\ne x_2$ define the same line if and only if
$x_{1}-x_{2}=x'_{1}-x'_{2}$ and $x_{1}^2-x_{2}^2={x'_{1}}^2-{x'_{2}}^2$,
that is $(x_{1},x_{2})=(x'_{1},x'_{2})$.
It follows that all the lines $\ell_{x_1,x_2}$ with $x_1\ne x_2$ are pairwise distinct and
the number of these lines is equal to $|A|^2-|A|\ll p^{2\alpha}$. We let $\mathcal{L}=\{\ell_{x_1,x_2}\,:\,
(x_1,x_2)\in A^2,\, x_1\ne x_2\}$.
By applying Lemma \ref{lem1}, we get for some $\gamma=\gamma(\alpha)>0$
$$
|\{ [(\xi_{1},\xi'_{2});\ell]\in C_{k}\times \mathcal{L}:\
(\xi_{1},\xi'_{2})\in \ell\}|\ll
|C_{k}|^{3/2-\gamma} +p^{(3-2\gamma)\alpha}.
$$
where $C_{k}$ is the set of couples
$(\xi_{1},\xi'_{2})\in\mathbb{F}_{p}^2$ such that the
number of different
couples $(x_{1},x_{2})\in A^2$ with $x_1\ne x_2$ satisfying
equation \eqref{equat8} with $\xi_{1}-x_{1}+x_{2}-t\in A$ is at least $k$.
Since there is a one-to-one correspondance between the
couples $(\xi_1,\xi'_2)\in C_k$ and the couples $(\xi_1,\xi_2)\in \mathbb{F}_p^2$
such that $\mu(\underline{\xi},t)\ge k$, we plainly have $|C_{k}|\le p^{3\alpha}/k$. Furthermore,
for fixed $(\xi_{1},\xi'_{2})$ in $\mathbb{F}_{p}^2$,
each choice of $x_{1}\in A$ gives at most two different
$x_{2}\in A$ such that \eqref{equat8} holds. Hence
$C_{k}$ is empty if $k>2|A|$.
We let $c_{k}=|C_{k}|$. We obtain
$$
c_{k}k\ll c_{k}^{3/2-\gamma} +p^{(3-2\gamma)\alpha},
$$
giving either
$$
c_{k}k\ll p^{(3-2\gamma)\alpha}
$$
or
$$
k\ll c_{k}^{1/2-\gamma}.
$$
Since $c_{k}\ll p^{3\alpha}/k$, the last bound is available only if
$$
k\le k(\alpha,\gamma):= cp^{(3-6\gamma)\alpha/(3-2\gamma)}, \quad
\text{for some constant $c>0$.}
$$
We have
\begin{equation*}
\sum_{\underline{\xi}\in\mathbb{F}_{p}^2}
\mu(\underline{\xi},t)^2=
\sum_{1\le k\le 2|A|}k^2(c_{k}-c_{k+1})
=\sum_{1\le k\le 2|A|}(2k-1)c_{k},
\end{equation*}
by partial summation. It follows that
\begin{align*}
\sum_{\underline{\xi}\in\mathbb{F}_{p}^2}\mu(\underline{\xi},t)^2
&=\sum_{1\le k\le k(\alpha,\gamma)} (2k-1)c_{k}
+\sum_{k(\alpha,\gamma)< k\le 2|A|} (2k-1)c_{k}\\
&\le 2\sum_{1\le k\le k(\alpha,\gamma)} p^{3\alpha}+
\sum_{k(\alpha,\gamma)< k\le 2|A|} p^{(3-2\gamma)\alpha}\\
&\ll p^{12(1-\gamma)\alpha/(3-2\gamma)}+  p^{(4-2\gamma)\alpha}\\
&\ll p^{(4-\gamma)\alpha}.
\end{align*}
By Cauchy-Schwarz inequality, we get
\begin{equation*}
\sum_{\underline{\xi}\in\mathbb{F}_{p}^2}\mu_2(\underline{\xi})^2=
\sum_{\underline{\xi}\in\mathbb{F}_{p}^2}\Big(\sum_{t\in A}
\mu(\underline{\xi},t)
\Big)^2\le |A|\sum_{t\in A}\sum_{\underline{\xi}\in\mathbb{F}_{p}^2}
\mu(\underline{\xi},t)^2
\le |A|^2\sup_{t\in A}\sum_{\underline{\xi}\in\mathbb{F}_{p}^2}
\mu(\underline{\xi},t)^2\ll p^{(6-\gamma)\alpha},
\end{equation*}
giving with \eqref{equatn1}
\begin{equation}\label{equat9}
\sum_{\underline{\xi}\in\mathbb{F}_{p}^2}\mu(\underline{\xi})^2
\le 2\sum_{\underline{\xi}\in\mathbb{F}_{p}^2}\mu_1(\underline{\xi})^2
+ 2\sum_{\underline{\xi}\in\mathbb{F}_{p}^2}\mu_2(\underline{\xi})^2
\ll p^{(6-\gamma)\alpha}.
\end{equation}
This yields for $\sum\mu(\underline{\xi})^2$
a sharper bound than that could be expected in general, namely $O(p^{6\alpha})$.

Returning to the estimation of $S'$, we obtain
\begin{align*}
|S'|^8 & \ll p^{17\alpha}
\sum_{z\in C'}
\sum_{\underline{\xi}\in\mathbb{F}_{p}^2}\mu(\underline{\xi})
\Big|
\sum_{\underline{\eta}\in\mathbb{F}_{p}^2}
\nu(\underline{\eta})\ee_{r}\big(a(z)\xi_{1}\eta_{1}+
b(z)\xi_{2}\eta_{2}\big)\Big|\\
&
\ll p^{17\alpha}\sum_{z\in C'}
\Big(\sum_{\underline{\xi}\in\mathbb{F}_{p}^2}
\mu(\underline{\xi})^2\Big)^{1/2}
\Big(\sum_{\underline{\xi}\in\mathbb{F}_{p}^2}
\Big|\sum_{\underline{\eta}\in\mathbb{F}_{p}^2}
\nu(\underline{\eta})\ee_{r}\big(a(z)\xi_{1}\eta_{1}+
b(z)\xi_{2}\eta_{2}\big)\Big|^2
\Big)^{1/2}
\end{align*}
which is
$$
\ll p^{17\alpha}\sum_{z\in C'}\Big(\sum_{\underline{\xi}\in\mathbb{F}_{p}^2}
\mu(\underline{\xi})^2\Big)^{1/2}
\Big(\sum_{\underline{\eta},\underline{\eta}'\in\mathbb{F}_{p}^2}
\nu(\underline{\eta})\nu(\underline{\eta}')
\sum_{\underline{\xi}\in\mathbb{F}_{p}^2}
\ee_{r}\big(a(z)\xi_{1}(\eta_{1}-\eta'_{1})+
b(z)\xi_{2}(\eta_{2}-\eta'_{2})\big)
\Big)^{1/2}
$$
by Cauchy-Schwarz inequality. For $z\in C'$, the summation over $\underline{\xi}$
is $p^2$ if $\underline{\eta}=\underline{\eta}'$ and $0$
otherwise. It follows that
$$
|S'|^{8}\ll p^{17\alpha+1}|C'|\Big(\sum_{\underline{\xi}\in\mathbb{F}_{p}^2}
\mu(\underline{\xi})^2\Big)^{1/2}
\Big(\sum_{\underline{\eta}\in\mathbb{F}_{p}^2}\nu(\underline{\eta})^2\Big)^{1/2}.
$$
By \eqref{equat6} and \eqref{equat9}, this yields
$$
|S'|^{8}\ll p^{(22-\gamma/2)\alpha+1}
$$
hence
\begin{equation}\label{equat10}
|S|\le |S_{0}|+|S'|\ll p^{((22-\gamma/2)\alpha+1)/8}.
\end{equation}
\end{proof}

We may mention that in the statement of Theorem \ref{thmNewSE},
$\gamma(\alpha)$ is a continuous function of $\alpha$.
As a corollary, we have

\begin{cor}\label{cor9}
Let $F$ as in the theorem.
Then the extractor defined by  $\mathrm{sgn}\sin \frac{2\pi F}p$ has exponential distribution
and entropy at most $1/2-\delta$, for some $\delta>0$.
\end{cor}

\begin{proof}
From Theorem \ref{thmNewSE}, we obtain that
$$
\max_{r\in\mathbb{F}_p\smallsetminus\{0\}}|S_r|\ll p^{3\alpha-\epsilon(\alpha)},
$$
where
\begin{equation}\label{equat11}
\epsilon(\alpha)=\frac{\alpha}8\Big(2+\frac{\gamma(\alpha)}2-\frac1{\alpha}\Big).
\end{equation}
Since $\gamma(1/2)>0$, we have $\epsilon(1/2)>0$, thus by continuity,
there exists $\delta>0$ such that for $\alpha>1/2-\delta$,
we have $\epsilon(\alpha)>0$.

The rest of the proof follows that in \cite{B}, namely we have
$$
\sum_{(x,y,z)\in A\times B\times C}\mathrm{sgn}\sin\left(\frac{2\pi  F(x,y,z)}{p}\right)
=
\sum_{r=1}^{p-1}c_rS_r + O(p^{3\alpha-1}),
$$
where the coefficients $c_r$ satisfy
$$
\mathrm{sgn}\sin\left(\frac{2\pi t}p\right)=\sum_{r=1}^{p-1}c_r \exp\left(\frac{2i\pi t}p\right)+O\Big(\frac1p\Big),
$$
and
$$
\sum_{r=1}^{p-1}|c_r|=O(\ln p).
$$
This gives
$$
\sum_{(x,y,z)\in A\times B\times C}\mathrm{sgn}\sin\left(\frac{2\pi  F(x,y,z)}{p}\right)
=O((\ln p)p^{3-\epsilon})
$$
and the corollary follows.
\end{proof}

\section{\bf Concluding remarks}\label{s5}

1. As indicated in section \ref{s3}, no function
of the type $F(x,y)=f(x)+g(y)$ or any translated of it is an expander.
Indeed let $I$ be an
interval with length $\asymp Cp^\alpha,$ ($0<\alpha <1, C>0$). By
the averaging argument there are $a$ and $b$ in
$\mathbb{F}_{p}$ such that
$$
|\{a+I\}\cap \{f(x):x\in\mathbb{F}_{p}\}|>C'p^\alpha,
$$
and
$$
|\{b+I\}\cap \{g(y):y\in\mathbb{F}_{p}\}|>C'p^\alpha,
$$ where $C'$ depends only on $C$ and the degree of $f$ and $g.$
Now let $A$ be the inverse image of $\{a+I\}\cap
\{f(x):x\in\mathbb{F}_{p}\}$ and let $B$ be the inverse image of
$\{b+I\}\cap \{g(y):y\in\mathbb{F}_{p}\}.$ Then the set $F(A,B)$ of all
elements of the form $F(x,y)$, $(x,y)\in A\times B$ is contained in
$a+b+2I$, hence the cardinality of $F(A,B)$ is at most a constant times the
cardinality of $A$ and $B$.

A similar argument  yields that no map of the kind $f(x)g(y)+c$
is an expander.

2. As quoted after Corollary \ref{cor9},
the functions $f_p(x,y)=\mathrm{sgn}\sin
F_p(x,y)$ give a 2-source extractor with entropy less than $1/2$,
if we let
 $F_p(x,y)=xy+x^2y^2$ or $F_p(x,y)=xy+g_p^{x+y},$ where $g_p$ is any generator
in $\mathbb{F}_{p}^{\times}.$ From the proof one can easily read that the
functions
\begin{equation}\label{e11}
xy+x^2h(y); \quad xh(y)+x^2y; \quad xy+x^2g_p^y; \quad xg_p^y+x^2y
\end{equation}
($h$ is any non-constant polynomial) induce also 2-source
extractors with entropy less than $1/2$ (see also
remark 4 below).

3. It is worth mentioning that for points and lines in $\mathbb{F}_{p}^2$,
 the bound given by the effective version of the Szemerédi-Trotter theorem
of \cite{LAV} is weaker then the trivial one in case
where the number $N$ of lines and points is less than $p$.
For this reason, it is seemingly not efficient for providing an effective
entropy less than $1/2$ for $k$-source extractor,
contrarily to Bourgain-Katz-Tao result which holds for
$p^{\varepsilon}<N<p^{2-\varepsilon}$.

4. Extractors are related to additive questions  in
$\mathbb{F}_{p}.$ In \cite{S} S\'ark\"ozy investigated the
following problem: let $A,B,C,D \subseteq \mathbb{F}_{p}$ be non-empty sets.
Then
the equation
$$
a+b=cd
$$
is solvable in $a\in A, b\in B, c\in C, d\in D$ provided
$|A||B||C||D|>p^3.$ This simple equation has many interesting
consequences. One can ask the more general question of investigating the
solvability of
\begin{equation}\label{e12}
a+b=F(c,d)
\end{equation}
where $F(x,y)$ is a two variables polynomial with integer
coefficients. Clearly the question is really interesting when we assume that
$|C|,|D|<\sqrt{p}.$

Let us say that $F(x,y)$ is an {\it essential} polynomial if
(under the condition $|C|,|D|<\sqrt{p}$) $|A||B|>p^2$ implies the
solvability of \eqref{e12}. So by the S\'ark\"ozy's result $F(x,y)=xy$
is an essential polynomial. From the proofs of propositions 3.6 and 3.7 of \cite{B},
it can be deduced that there exist $\delta>0$ and $\epsilon>0$ such that for any $r\in\mathbb{F}_p\smallsetminus\{0\}$ and for any $C,D\subset\mathbb{F}_p$ with
$|C|,|D|>p^{1/2-\delta}$,
\begin{equation}\label{equat100}
\Big|\sum_{c\in C,d\in D}\ee_r(F_p(c,d))\Big|=O(|C||D|p^{-\epsilon}),
\end{equation}
where $F=(F_p)_p$ is any one of the following families of functions:

- $F_p(x,y)=x^{1+u}y+x^{2-u}h(y)$ for any $p$, where we fix $u\in\{0,1\}$ and any non constant polynomial $h(y)\in\mathbb{Z}[y]$.

- $F_p(x,y)=x^{1+u}y+x^{2-u}g_p^y$ for any $p$ where $g_p$ generates $\mathbb{F}_p^{\times}$ and $u\in\{0,1\}$ is fixed.

This yields  the following result:

\begin{prop}
Let $(F_p)_p$ be one of the two families of functions defined above.
There exist real numbers $0<\delta,\delta'<1$ such that for any $p$ and for
any sets $A,B,C,D \subseteq \mathbb{F}_{p}$ fulfilling the
conditions
$$
|C|>p^{1/2-\delta}, \quad |D|>p^{1/2-\delta} \quad
|A||B|>p^{2-\delta'},
$$
there exist  $a\in A, b\in B, c\in C, d\in D$ solving
the equation
\begin{equation}\label{equat101}
a+b=F_p(c,d).
\end{equation}
\end{prop}
\begin{proof}[Sketch of the proof]
Let $N$ be the number of solutions of \eqref{equat101}. Then by
following S\'ark\"ozy's
argument and using the bound \eqref{equat100}, we obtain
$$
\Big|N-\frac{|A||B||C||D|}p\Big|\ll |A|^{1/2}|B|^{1/2}|C||D|p^{-\epsilon},
$$
which gives the result for $p$ large enough with $\delta'=\epsilon$.
For $p\le p_0$, it suffices to reduce
$\delta'$ in order to have also $p_0^{2-\delta'}\ge p_0^2-1$, and the result becomes
trivial since $|A||B|>p^{2-\delta'}$ implies either $A=\mathbb{F}_p$ or
$B=\mathbb{F}_p$.
\end{proof}

\bigskip

5. Note that the range of our function
$F(x,y,z)=a(z)xy+b(z)x^2g(y)+h(y,z)$ studied in section \ref{newsection}
is well-spaced i.e. the set $F(A,B,C)$
of elements of $\mathbb{F}_{p}$ of the form $F(x,y,z)$ where
$(x,y,z)\in A\times B\times C$,
intersects every not too long interval, provided the cardinalities
of the sets are $\asymp p^{\alpha}$ with $\alpha>1/2-\delta.$

The bound we obtain for the exponential sum $S$ in the proof
of theorem \ref{thmNewSE} yields the following result:

\begin{cor}
Let $\epsilon(\alpha)$ given by \eqref{equat11} and $\delta$ given in Corollary \ref{cor9}.
Let $L_1\le L_2$ be arbitrary positive real numbers, $F(x,y,z)\in\mathbb{Z}[x,y,z]$ as in theorem \ref{thmNewSE} and $A,B,C$ be subsets of
$\mathbb{F}_{p}$ with $L_1 p^{\alpha}\le |A|, |B|, |C|\le L_2p^{\alpha}$ where
$\alpha > 1/2-\delta$. Then $F(A,B,C)$ intersects every interval
$[u+1,u+L]$ in $\mathbb{F}_p$ provided
$L\gg  p^{1-\epsilon(\alpha)}$ where the implied constant depends only on
$F$, $L_1$ and $L_2$.
\end{cor}

For seek of completeness we include the proof.

\begin{proof}
Let $S(w)$ be the number of triples $(a,b,c)\in F(A,B,C)$ such that
$w=F(a,b,c)$. Let
$I=[1,L/2]$ and denote by  $I(w)$ its indicator. Then
$F(A,B,C)\cap [u+1,u+L]$ is not empty if and only if the real sum
$$
T=\sum_w S(w-u)I*I(-w)
$$
is not zero.
Denote the Fourier transform of the indicators of $S$ resp. $I$ by
$S_r$ resp. $I_r.$ By the Fourier inversion formula we have
$$
T=\frac1p\sum_r S_r{\overline{I_r^2}}\ee_r(-u) \ge
\frac{S_0I^2_0}p-\frac1p\sum_{r\neq 0} |S_r||I_r|^2= \frac1p
|A||B||C|I_0^2-\frac1p\sum_{r\neq 0} |S_r||I_r|^2.
$$
By the triangle inequality, the non trivial upper bound for $|S_r|$ when $r\ne0$ and by the
Parseval formula, \eqref{equat10} and \eqref{equat11} we get
\begin{equation*}
\Big|T-{1\over p}|A||B||C|I_0^2\Big|\leq \frac1p\sum_{r\neq 0} |S_r||I^2_r|\\
\le \frac1p\max_{r\ne0} |S_r|\sum_{r}|I^2_r|\ll
 p^{3\alpha-\epsilon(\alpha)} I_0.
\end{equation*}
Hence the set $F(A,B,C)\cap [u+1,u+L]$ is not empty if
$$
\frac1p|A||B||C|I_0\gg  p^{3\alpha-\epsilon(\alpha)}
$$
or equivalently if
$$
L\gg p^{1-\epsilon(\alpha)},
$$
as asserted.
\end{proof}


\vskip 1cm

\begin{thebibliography}{9}

\bibitem{BKSSW} Barak, B., Kindler G., Shaltiel R., Sudakov B. and Wigderson A.,
Simulating independence: new constructions of condensers, Ramsey graphs, dispersers, and Extractors,
Proceeding of the 37th annual ACM Symposium on Theory of Computing.

\bibitem{V} Bombieri, E.,  On exponential sums in finite fields,  Amer. J. Math. {\bf88}  (1966),  71--105

\bibitem{BKT} Bourgain J., Katz N. and Tao T., A sum-product theorem in finite fields and application,
Geom. Funct. Anal. {\bf14} (2004), 27--57.
\bibitem{B} Bourgain, J., More on the sum-product phenomenon in prime fields and its application,
Int. J. of Number Theory {\bf1} (2005), 1--32.
\bibitem{E2} Erd\H os P., Some remarks on number theory, Riveon
Lematematika {\bf9} (1944), 45--48.
\bibitem{E3} Erd\H os P., An asymptotic inequality in the theory
of numbers, Vestnik Leningrad Univ. {\bf15} no 13 (1960), 41--49
(in Russian).
\bibitem{S} S\'ark\"ozy, A., On sums and products of residues
modulo $p$, Acta Arith. {\bf118} (2005), 403--409.
\bibitem{T} Tenenbaum G., Sur la probabilit\'e qu'un entier
poss\`ede un diviseur dans un intervalle donn\'e, Compositio Math.
{\bf51} (1984), 243--263.
\bibitem{LAV}  Vinh L.A., Szemer\'edi--Trotter type theorem and sum-product
estimate in finite fields, arXiv:0711.4427v1[CO].
\end{thebibliography}
\end{document}